\documentclass[12pt,english]{article}
\usepackage[T1]{fontenc}
\usepackage[latin9]{inputenc}
\usepackage[letterpaper]{geometry}
\usepackage{textcomp}
\usepackage{amsthm}
\usepackage{amsmath}
\usepackage{amssymb}
\usepackage{setspace}
\usepackage[authoryear]{natbib}
\makeatletter

\providecommand{\tabularnewline}{\\}

\newcommand{\lyxaddress}[1]{
\par {\raggedright #1
\vspace{1.4em}
\noindent\par}
}
  \theoremstyle{plain}
  \newtheorem{lem}{\protect\lemmaname}
\theoremstyle{plain}
\newtheorem{thm}{\protect\theoremname}
  \theoremstyle{plain}
  \newtheorem{cor}{\protect\corollaryname}
  \theoremstyle{remark}
  \newtheorem{rem}{\protect\remarkname}
 \theoremstyle{definition}
  \newtheorem{example}{\protect\examplename}

\DeclareMathOperator{\unif}{U}

\DeclareMathOperator{\N}{N}

\DeclareMathOperator{\pri}{prior}

\makeatother

\usepackage{babel}
  \providecommand{\examplename}{Example}
  \providecommand{\lemmaname}{Lemma}
  \providecommand{\remarkname}{Remark}
\providecommand{\corollaryname}{Corollary}
\providecommand{\theoremname}{Theorem}

\begin{document}

\title{\textbf{\large ~}\\
Controlling the degree of caution in statistical inference with the
Bayesian and frequentist approaches as opposite extremes\\
\textbf{~}}

\maketitle
~~\\
David R. Bickel

\lyxaddress{Ottawa Institute of Systems Biology\\
Department of Biochemistry, Microbiology, and Immunology\\
University of Ottawa; 451 Smyth Road; Ottawa, Ontario, K1H 8M5}
\begin{abstract}
In statistical practice, whether a Bayesian or frequentist approach
is used in inference depends not only on the availability of prior
information but also on the attitude taken toward partial prior information,
with frequentists tending to be more cautious than Bayesians. The
proposed framework defines that attitude in terms of a specified amount
of caution, thereby enabling data analysis at the level of caution
desired and on the basis of any prior information. The caution parameter
represents the attitude toward partial prior information in much the
same way as a loss function represents the attitude toward risk. When
there is very little prior information and nonzero caution, the resulting
inferences correspond to those of the candidate confidence intervals
and p-values that are most similar to the credible intervals and hypothesis
probabilities of the specified Bayesian posterior. On the other hand,
in the presence of a known physical distribution of the parameter,
inferences are based only on the corresponding physical posterior.
In those extremes of either negligible prior information or complete
prior information, inferences do not depend on the degree of caution.
Partial prior information between those two extremes leads to intermediate
inferences that are more frequentistic to the extent that the caution
is high and more Bayesian to the extent that the caution is low.
\end{abstract}

\section{\label{sec:Introduction}Introduction}

The controversy between Bayesianism and frequentism may be irresolvable
to the extent that it reflects honest differences in personal attitudes
of statisticians rather than differences in their knowledge or rationality.
As \citet{Efron2005} pointed out,
\begin{quote}
The Bayesian-frequentist debate reflects two different attitudes about
the process of doing science, both quite legitimate. Bayesian statistics
is well suited to individual researchers, or a research group, trying
to use all of the information at its disposal to make the quickest
possible progress. In pursuing progress, Bayesians tend to be aggressive
and optimistic with their modeling assumptions. Frequentist statisticians
are more cautious and defensive. One definition says that a frequentist
is a Bayesian trying to do well, or at least not too badly, against
any possible prior distribution. The frequentist aims for universally
acceptable conclusions, ones that will stand up to adversarial scrutiny. 
\end{quote}
On one hand, methodology reflecting extreme caution in the form of
the minimax-like attitude attributed to frequentists and, on the other
hand, methodology reflecting the extreme reliance on modeling assumptions
attributed to Bayesians both play useful roles in statistical inference.
Building on that premise, the idea motivating this paper is that methodology
for moderate amounts of caution also has a place in practical data
analysis. The extent of such caution will be formally defined in order
to facilitate making statistical inferences at the level of caution
appropriate to the situation. 

The formal definition will build on previous work to formalize caution
in the face of uncertainty. Attitudes toward uncertainty have long
been mathematically modeled in the economics literature. \citet{Ellsberg1961b}
identified two distinct types of uncertainty:\emph{ risk} is the variability
in an unknown quantity that threatens assets, whereas \emph{ambiguity}
is ignorance about the extent of such variability. The same agent
may be much more cautious toward risk than toward ambiguity or vice
versa. A utility or loss function can model an agent's attitude toward
risk but not its attitude toward ambiguity. Because frequentist actions
can differ from Bayesian actions given the same loss function, the
ambiguity attitude is much more pertinent than the risk attitude to
the concept of caution needed to represent and balance the two basic
approaches to statistical inference.

\citet{Ellsberg1961b} distinguished {}``pessimism'' from {}``conservatism'':
the former is an excessive belief that worst-case scenarios will materialize,
whereas the latter only involves cautiously acting as if they will.
In other words, the attitude of {}``hoping for the best, preparing
for the worst'' is consistent with conservatism but not pessimism.
While that attitude does motivate much of frequentist statistics,
{}``conservatism'' already has technical meanings in the statistics
literature, e.g., conservative confidence intervals have higher-than-nominal
coverage rates. For that reason, the term {}``caution'' will be
used when assigning an operational definition to the degree of conservatism
toward ambiguity in the sense of \citet{Ellsberg1961b}. 

In an example from \citet{Ellsberg1961b}, a ball is randomly drawn
from an urn of 90 balls, each of one of three possible colors: red,
black, and yellow. Nothing is known about the distributions of the
balls in the urn except that exactly 30 are red. Thus, there is ambiguity
in the distribution of black and yellow balls. The agent would gain
a reward of \$0 or \$100 based on its taking action I or action II
according to utility function displayed as Table \ref{tab:Utility1}
in setting 1. In setting 2, the agent would instead gain \$0 or \$100
based on its taking action III or action IV according to the utility
function displayed as Table \ref{tab:Utility2}. Agents cautious toward
ambiguity would choose action I over action II in setting 1 but would
take action IV over action III in setting 2, against subjective Bayesian
concepts of coherence but without requiring the extreme caution of
a minimax strategy \citep{Ellsberg1961b}. 

\begin{table}
\begin{tabular}{|c|c|c|c|}
\hline 
 & Red drawn & Black drawn & Yellow drawn\tabularnewline
\hline 
\hline 
action I & \$100 & \$0 & \$0\tabularnewline
\hline 
action II & \$0 & \$100 & \$0\tabularnewline
\hline 
\end{tabular}

\caption{\label{tab:Utility1}Utility function for actions I-II and the three
possible states of nature.}
\end{table}
\begin{table}
\begin{tabular}{|c|c|c|c|}
\hline 
 & Red drawn & Black drawn & Yellow drawn\tabularnewline
\hline 
\hline 
action III & \$100 & \$0 & \$100\tabularnewline
\hline 
action IV & \$0 & \$100 & \$100\tabularnewline
\hline 
\end{tabular}

\caption{\label{tab:Utility2}Utility function for actions III-IV and the three
possible states of nature.}
\end{table}

In the absence of ambiguity, the axiomatic system of \citet[§3.6]{MaxUtility1944}
and later generalizations prescribe choosing the action that maximizes
expected utility. By forcefully applying such a system to conditional
expectations given observed data, \citet{RefWorks:126} revitalized
Bayesian statistics. The action that maximizes expected utility with
respect to a Bayesian posterior is called the \emph{posterior Bayes
action}. Ambiguity about the posterior is usually modeled in terms
of a set $\dot{\mathcal{P}}$ of multiple posteriors in place of a
single posterior. A multiplicity of posteriors may arise from insufficient
elicitation of subjective prior opinions \citep[e.g.,][]{RefWorks:16},
from a spread in a gamble's buying and selling prices \citep[e.g.,][]{RefWorks:Walley1991},
or, more objectively, from ignorance as to which prior distribution
in a set describes the physical variability of a parameter. The last
source accords best with the notion of ambiguity as used in \citet{Ellsberg1961b},
\citet{RefWorks:1612}, \citet{JAFFRAY1989}, and \citet{RefWorks:1618}. 

In the Bayesian statistics literature, the most studied decision-theoretic
approach for sets of priors is the (marginal) $\Gamma$-minimax strategy
\citep[e.g.,][]{RefWorks:16}, which formulates the problem in terms
of minimax risk in the frequentist sense of \citet{Wald1950}. The
closely related conditional $\Gamma$-minimax strategy \citep[e.g.,][]{Betro1992b}
takes the action that minimizes the expected loss maximized over all
of the posterior distributions in a set $\dot{\mathcal{P}}$, each
member of which corresponds to a prior distribution in a set traditionally
denoted by $\Gamma$. That statistical strategy is a special case
of the maxmin expected utility strategy \citep{Hurwicz1951a,ISI:A1989AL55500003},
which takes the action that maximizes the expected utility minimized
over a set of distributions. Both {}``robust Bayesian'' strategies
are reviewed in \citet{citeulike:7129735}.

The following equation extends the conditional $\Gamma$-minimax strategy
to the problem of conducting statistical inference at a specified
degree of \emph{caution} $\kappa$ and with respect to a Bayesian
posterior $\dot{P}\in\dot{\mathcal{P}}$ that is not generally the
true physical distribution of the parameter $\theta$. For any $\kappa\in\left[0,1\right]$,
the \emph{$\kappa$-conditional-$\Gamma$ ($\kappa$CG) action }is
defined as 
\begin{equation}
\dot{a}_{\kappa}=\arg\inf_{a\in\mathcal{A}}\left(\kappa\sup_{P^{\prime}\in\dot{\mathcal{P}}}\int L\left(\theta,a\right)dP^{\prime}\left(\theta\right)+\left(1-\kappa\right)\int L\left(\theta,a\right)d\dot{P}\left(\theta\right)\right),\label{eq:CG}
\end{equation}
with the conventions that $\kappa\times\infty=0$ if $\kappa=0$ and
$\left(1-\kappa\right)\times\infty=0$ if $\kappa=1$. The $\kappa$CG
action reduces to the conditional $\Gamma$-minimax action under complete
caution $\left(\kappa=1\right)$ and to the posterior Bayes action
in the complete absence of caution $\left(\kappa=0\right)$. For discrete
$\theta$, this $\kappa$ is isomorphic to quantities used by \citet{Ellsberg1961b},
\citet{RefWorks:1618}, and \citet{TAPKING2004} and is similar in
spirit to the quantity in \citet{Hurwicz1951} and \citet{RefWorks:1612}
that \citet{ISI:000173555100002} calls {}``caution.'' \citet{RefWorks:1618}
stressed the equivalent of the rearrangement of equation \eqref{eq:CG}
as
\begin{equation}
\dot{a}_{\kappa}=\arg\inf_{a\in\mathcal{A}}\left(\sup_{P\in\left\{ \kappa P^{\prime}+\left(1-\kappa\right)\dot{P}:P^{\prime}\in\dot{\mathcal{P}}\right\} }\int L\left(\theta,a\right)dP\left(\theta\right)\right).\label{eq:CG-mixture}
\end{equation}

The $\kappa$CG strategy has two drawbacks that will prevent its use
in many applications. First, under standard loss functions, the conditional
$\Gamma$-minimax (1CG) strategy requires either that $\dot{\mathcal{P}}$
impose strict bounds on the parameter space \citep{AbdollahBayati2011}
or that $\mathcal{A}$ be severely restricted \citep{Betro1992b},
and the $\kappa$CG strategy with $0<\kappa<1$ has the same limitation.
Second, since the 1CG strategy is not necessarily a frequentist procedure,
the $\kappa$CG framework does not fulfill the above goal of formulating
procedures that reduce to frequentist procedures given complete caution.

Following the preliminary notation and definitions of Section \ref{sec:preliminaries},
an information-theoretic framework will be introduced in Section \ref{sec:General-theory}
to overcome the identified limitations of the $\kappa$CG framework.
Simple examples demonstrating the wide applicability of the information-theoretic
framework will appear in Section \ref{sec:Examples}. Further generality
will be carried out in Section \ref{sec:Variations} by exchanging
roles of frequentist and Bayesian procedures and by noting the particular
applications that call for each role exchange. Section \ref{sec:Discussion}
closes the paper with a brief discussion.

\section{\label{sec:preliminaries}Bayesian and frequentist posterior distributions}

\subsection{\label{sub:Preliminary-concepts}Preliminary concepts}

The observed data vector $x\in\mathcal{X}$ is modeled as a realization
of a random variable $X$ of probability space $\left(\mathcal{X},\mathfrak{X},P_{\theta_{\ast},\lambda_{\ast}}\right)$,
which for some parameter set $\Theta_{\ast}\times\Lambda_{\ast}$
is indexed by an \emph{interest parameter} $\theta_{\ast}\in\Theta_{\ast}$
and potentially also by a \emph{nuisance parameter} $\lambda_{\ast}\in\Lambda_{\ast}$.
The family $\left\{ P_{\theta_{\ast},\lambda_{\ast}}:\theta_{\ast}\in\Theta_{\ast},\lambda_{\ast}\in\Lambda_{\ast}\right\} $
will be called the \emph{sampling model} for $x$.

Inferences will be made about the \emph{focus parameter} $\theta=\theta\left(\theta_{\ast}\right)$,
a subparameter of the interest parameter, in a set $\Theta$. In the
simplest case, $\theta=\theta_{\ast}$ and $\Theta=\Theta_{\ast}$,
but there are many other possibilities. For example, when testing
the null hypothesis that $\theta_{\ast}=0$ against the alternative
hypothesis that $\theta_{\ast}\ne0$ for $\Theta_{\ast}=\mathbb{R}$,
it is convenient to define the focus parameter by $\theta=0$ if $\theta_{\ast}=0$
and $\theta=1$ if $\theta_{\ast}\ne0$, in which case $\Theta=\left\{ 0,1\right\} $.
Let $\mathcal{H}$ denote a $\sigma$-field that allows any physically
meaningful hypothesis about $\theta$ to be expressed as {}``$\theta$
is in $\Theta^{\dagger}$,'' where $\Theta^{\dagger}\in\mathcal{H}$.

\subsection{\label{sub:Bayesian-posteriors}Bayesian posteriors}

In the Bayesian setting, the above sampling model is understood as
conditional on the parameter values with respect to some prior distribution,
as follows. Every member $P_{\ast}^{\pri}$ of some set $\mathcal{P}_{\ast}^{\pri}$
is a distribution such that there is a random triple $\left\langle \dot{X},\dot{\theta}_{\ast},\dot{\lambda}_{\ast}\right\rangle \sim P_{\ast}^{\pri}$
and such that $P_{\theta_{\ast},\lambda_{\ast}}=P_{\ast}^{\pri}\left(\bullet\vert\dot{\theta}_{\ast}=\theta_{\ast},\dot{\lambda}_{\ast}=\lambda_{\ast}\right)$
for all $\theta_{\ast}\in\Theta_{\ast},\lambda_{\ast}\in\Lambda_{\ast}$.

Let $\mathcal{P}$ denote the set of all probability distributions
on $\left(\Theta,\mathcal{H}\right)$. Before observing data, knowledge
about the focus parameter is represented by $\mathcal{P}^{\pri}$,
the set of all distributions of $\left\langle \dot{X},\theta\left(\dot{\theta}_{\ast}\right)\right\rangle $
 for all $P_{\ast}^{\pri}\in\mathcal{P}_{\ast}^{\pri}$:
\[
\mathcal{P}^{\pri}=\left\{ P^{\pri}\in\mathcal{P}:\left\langle \dot{X},\theta\left(\dot{\theta}_{\ast}\right)\right\rangle \sim P^{\pri},\left\langle \dot{X},\dot{\theta}_{\ast},\dot{\lambda}_{\ast}\right\rangle \sim P_{\ast}^{\pri}\in\mathcal{P}_{\ast}^{\pri}\right\} .
\]
The marginal distribution of each $\theta\left(\dot{\theta}_{\ast}\right)$
is in $\mathcal{P}$ and is called a \emph{plausible prior }distribution
since it is consistent with pre-data knowledge. 

The Bayesian approach yields inferences about the focus parameter
on the basis of a single distribution, $\dot{P}^{\pri}\in\mathcal{P}^{\pri}$.
If $\left\langle \dot{X},\dot{\theta}\right\rangle \sim\dot{P}^{\pri}$,
then the \emph{working prior} distribution is the marginal distribution
of $\dot{\theta}$. It follows that the working prior is one of
the plausible priors.

The \emph{working Bayesian posterior} $\dot{P}$ and the \emph{knowledge
base} $\dot{\mathcal{P}}$ \citep{ISI:000224212800044} are defined
such that 
\[
\dot{P}=\dot{P}^{\pri}\left(\bullet\vert\dot{X}=x\right);
\]
\[
\dot{\mathcal{P}}=\left\{ P^{\prime}\left(\bullet\vert\dot{X}=x\right):P^{\prime}\in\mathcal{P}^{\pri}\right\} .
\]
$\dot{P}$ is simply the Bayesian posterior distribution corresponding
to the working prior, and $\dot{\mathcal{P}}$ is likewise the set
of Bayesian posteriors in $\mathcal{P}$ that correspond to plausible
prior distributions. To prevent confusion with $\dot{P}$, members
of $\dot{\mathcal{P}}$ will be referred to as \emph{plausible posteriors}
since they are the parameter distributions consistent with the mathematical
representation either of a physical system or of a belief system \citep[cf.][]{ISI:A1979GQ08900002,ISI:000224212800044}.
Thus, the posterior that would be used in purely Bayesian inference
is one of the plausible posteriors $\left(\dot{P}\in\dot{\mathcal{P}}\right)$.

\subsection{\label{sub:Confidence-posteriors}Confidence posteriors}

The sampling model of Section \ref{sub:Preliminary-concepts} admits
not only system constraints and Bayesian inference (§\ref{sub:Bayesian-posteriors})
but also frequentist inference in the form of confidence intervals
and p-values. Let $\mathcal{H}_{\ast}$ denote $\mathcal{B}\left(\Theta_{\ast}\right)$,
the Borel set of $\Theta_{\ast}$. Given some $\alpha\in\left[0,1\right]$,
if the function $\widehat{\Theta}\left(1-\alpha,\bullet\right):\mathcal{X}\rightarrow\mathcal{H}_{\ast}$
satisfies

\begin{equation}
P_{\theta_{\ast},\lambda_{\ast}}\left(\theta_{\ast}\in\widehat{\Theta}\left(1-\alpha;X\right)\right)=1-\alpha\label{eq:confidence-set}
\end{equation}
for all $\theta_{\ast}\in\Theta_{\ast}$ and $\lambda_{\ast}\in\Lambda_{\ast}$,
then $\widehat{\Theta}\left(1-\alpha;X\right)$ is called a $100\left(1-\alpha\right)\%$
\emph{confidence set} for $\theta_{\ast}$. Suppose $\widehat{\Theta}:\left[0,1\right]\times\mathcal{X}\rightarrow\Theta_{\ast}$
defines nested confidence sets in the sense that $\widehat{\Theta}\left(1-\alpha;X\right)$
is a $100\left(1-\alpha\right)\%$ confidence set for $\theta_{\ast}$
given any confidence level $1-\alpha\in\left[0,1\right]$ and such
that 
\[
0\le1-\alpha_{1}<1-\alpha_{2}\le1\implies\widehat{\Theta}\left(1-\alpha_{1};X\right)\subset\widehat{\Theta}\left(1-\alpha_{2};X\right)
\]
almost surely. A \emph{confidence posterior distribution} is a distribution
$\ddot{P}_{\ast}$ on $\left(\Theta_{\ast},\mathcal{H}_{\ast}\right)$
for which
\begin{equation}
\ddot{P}_{\ast}\left(\Theta^{\dagger}\right)=\ddot{P}_{\ast}\left(\ddot{\theta}_{\ast}\in\Theta^{\dagger}\right)\in\left\{ 1-\alpha:\widehat{\Theta}\left(1-\alpha;x\right)=\Theta^{\dagger},0\le\alpha\le1\right\} \label{eq:confidence-level}
\end{equation}
for all $x\in\mathcal{X}$ and $\Theta^{\dagger}\in\mathcal{H}_{\ast}$,
where $\ddot{\theta}_{\ast}$ is a random variable of distribution
$\ddot{P}_{\ast}$. \citet{Polansky2007b} called $\ddot{P}_{\ast}\left(\Theta^{\dagger}\right)$
the \emph{observed confidence level} of the hypothesis that $\theta_{\ast}\in\Theta^{\dagger}$.
Confidence posteriors for which $\theta_{\ast}$ is a real scalar
$\left(\Theta_{\ast}\subseteq\mathbb{R}\right)$ and the $\sigma$-field
is Borel $\left(\mathcal{H}_{\ast}=\mathcal{B}\left(\Theta_{\ast}\right)\right)$
are usually called \emph{confidence distributions}, each of which
encodes confidence intervals of all confidence levels and hypothesis
tests of all simple null hypotheses \citep{Efron19933}. Various devices
extend confidence posteriors to cases in which their posterior probabilities
only approximately match confidence levels \citep{RefWorks:127,RefWorks:130,Polansky2007b,CoherentFrequentism}.

The identity between confidence posterior probabilities and levels
of confidence \eqref{eq:confidence-level} clears up the misunderstanding
that confidence levels and p-values cannot be interpreted as epistemological
probabilities of hypotheses given the observed data. In fact, since
$\ddot{P}_{\ast}$ is a Kolmogorov probability measure on parameter
space, decisions made using various loss functions by the confidence
posterior action
\[
\ddot{a}=\arg\inf_{a\in\mathcal{A}}\int L\left(\theta_{\ast},a\right)d\ddot{P}_{\ast}\left(\theta_{\ast}\right)
\]
for each loss function $L$ are coherent with each other in the senses
usually associated with Bayesian inference, whether or not $\ddot{P}_{\ast}$
can be derived from some prior via Bayes's theorem \citep{conditional2009,CoherentFrequentism}.

Let $\ddot{\mathcal{P}}_{\ast}$ denote the set of confidence posteriors
on $\left(\Theta_{\ast},\mathcal{H}_{\ast}\right)$ that are under
consideration. For example, $\ddot{\mathcal{P}}_{\ast}$ could be
the set of a single confidence posterior, the set of all distributions
on $\left(\Theta_{\ast},\mathcal{H}_{\ast}\right)$ that satisfy equation
\eqref{eq:confidence-level}, or, as in \citet{CoherentFrequentism},
the set of two approximate confidence posteriors or the convex set
of all mixtures of the two.

The set $\ddot{\mathcal{P}}$ will represent the set of distributions
of $\theta\left(\ddot{\theta}_{\ast}\right)$ for all $\ddot{P}_{\ast}\in\ddot{\mathcal{P}}_{\ast}$:
\[
\ddot{\mathcal{P}}=\left\{ P^{\prime\prime}\in\mathcal{P}:\theta\left(\ddot{\theta}_{\ast}\right)\sim P^{\prime\prime},\ddot{\theta}_{\ast}\sim\ddot{P}_{\ast}\in\ddot{\mathcal{P}}_{\ast}\right\} .
\]
Thus, for any $\ddot{P}_{\ast}\in\ddot{\mathcal{P}}_{\ast}$, there
is a random parameter $\ddot{\theta}=\theta\left(\ddot{\theta}_{\ast}\right)$
of distribution $\ddot{P}\in\ddot{\mathcal{P}}$ such that $\ddot{P}\left(\ddot{\theta}\in\left\{ \theta\left(\theta_{\ast}^{\prime\prime}\right):\theta_{\ast}^{\prime\prime}\in\Theta_{\ast}^{\dagger}\right\} \right)=\ddot{P}_{\ast}\left(\ddot{\theta}_{\ast}\in\Theta^{\dagger}\right)$
for all $\Theta_{\ast}^{\dagger}\in\mathcal{H}_{\ast}$. $\ddot{\mathcal{P}}$
will be considered as a set of confidence posterior distributions
of the focus parameter even though more literally they are not necessarily
confidence posteriors but rather fiducial-like distributions derived
from the set $\ddot{\mathcal{P}}_{\ast}$ of confidence posteriors
by the laws of probability. (\citet{RefWorks:1175} provides a recent
review of fiducial inference.) In the simplest case of $\ddot{\theta}=\ddot{\theta}_{\ast}$,
$\left(\Theta,\mathcal{H}\right)=\left(\Theta_{\ast},\mathcal{H}_{\ast}\right)$
and $\ddot{\mathcal{P}}=\ddot{\mathcal{P}}_{\ast}$. While confidence
distributions are used here for concreteness, $\ddot{\mathcal{P}}$
can be a set of any distributions on $\left(\Theta,\mathcal{H}\right)$
to use as benchmarks with respect to which the posterior introduced
in the next section is intended as an improvement.

\section{\label{sec:General-theory}Framework of moderate inference}

\subsection{Moderate posteriors}

Let $P$ and $Q$ denote probability distributions on $\left(\Theta,\mathcal{H}\right)$.
The \emph{information divergence} \emph{of $P$ with respect to $Q$}
is defined as
\begin{equation}
I\left(P||Q\right)=\int dP\log\left(\frac{dP}{dQ}\right)\label{eq:cross-entropy}
\end{equation}
if $Q$ is absolutely continuous with respect to $P$ and $I\left(P||Q\right)=\infty$
if not, where $0\log\left(0\right)=0$ and $0\log\left(0/0\right)=0$.
$I\left(P||Q\right)$ goes by many names in literature, including
{}``Kullback-Leibler information'' and {}``cross entropy.'' Viewing
$I\left(P||Q\right)$ as information leads to the concept of how much
information for statistical inference would be gained by replacing
a confidence posterior $P^{\prime\prime}\in\ddot{\mathcal{P}}$ with
another posterior $Q\in\mathcal{P}$ if the plausible posterior $P^{\prime}\in\dot{\mathcal{P}}$
were the physical distribution of the parameter $\theta$. Specifically,
\[
I\left(P^{\prime}||P^{\prime\prime}\rightsquigarrow Q\right)=I\left(P^{\prime}||P^{\prime\prime}\right)-I\left(P^{\prime}||Q\right),
\]
as a special case of {}``information gain'' \citep{Phuber77}, is
called the \emph{inferential gain} \emph{of} $Q$ \emph{relative to}
$P^{\prime\prime}$ \emph{given} $P^{\prime}$ \citep{continuum}.
(The $\rightsquigarrow$ notation is borrowed from \citet{Topsoe2007b}.)

In analogy with equation \eqref{eq:CG-mixture}, the \emph{caution}
$\kappa\in\left[0,1\right]$ is then the extent to which a {}``worst-case''
plausible posterior $P^{\prime}\in\dot{\mathcal{P}}$ is used for
inference as opposed to the working Bayesian posterior $\dot{P}$
in this definition of the \emph{$\kappa$-inferential gain of }$Q$
\emph{relative to} $P^{\prime\prime}$ \emph{given} $P^{\prime}$
\emph{and }$\dot{P}$:

\begin{equation}
I\left(P^{\prime},\dot{P}||P^{\prime\prime}\rightsquigarrow Q;\kappa\right)=I\left(\kappa P^{\prime}+\left(1-\kappa\right)\dot{P}||P^{\prime\prime}\rightsquigarrow Q\right).\label{eq:inferential-utility}
\end{equation}
The posterior distribution that has the highest $\kappa$-inferential
gain in the following sense will be used for making inferences and
decisions. The \emph{moderate posterior distribution with caution
$\kappa$ relative to} $\ddot{\mathcal{P}}$ \emph{given} $\dot{\mathcal{P}}$
\emph{and }$\dot{P}$ is denoted by $\widetilde{P}_{\kappa}$ and
defined by
\begin{equation}
\inf_{P^{\prime\prime}\in\ddot{\mathcal{P}}}\inf_{P^{\prime}\in\dot{\mathcal{P}}}I\left(P^{\prime},\dot{P}||P^{\prime\prime}\rightsquigarrow\widetilde{P}_{\kappa};\kappa\right)=\inf_{P^{\prime\prime}\in\ddot{\mathcal{P}}}\sup_{Q\in\mathcal{P}}\inf_{P^{\prime}\in\dot{\mathcal{P}}}I\left(P^{\prime},\dot{P}||P^{\prime\prime}\rightsquigarrow Q;\kappa\right).\label{eq:moderate-posterior}
\end{equation}
Less technically, $\widetilde{P}_{\kappa}$ is the posterior distribution
that maximizes the worst-case inferential gain relative to the confidence
posterior $P^{\prime\prime}$, which is in turn chosen to minimize
the maximum worst-case gain. In the case that equation \eqref{eq:moderate-posterior}
does not have a unique solution, the moderate posterior is defined
to be as close as possible to the working Bayesian posterior: 
\begin{equation}
\widetilde{P}_{\kappa}=\arg\inf_{P^{\prime\prime\prime}\in\widetilde{\mathcal{P}}_{\kappa}}I\left(\dot{P}||P^{\prime\prime\prime}\right),\label{eq:moderate-posterior-unique}
\end{equation}
where the set $\widetilde{\mathcal{P}}_{\kappa}$ of \emph{candidate
moderate posteriors} is defined as the set of all distributions in
$\mathcal{P}$ such that every member of $\widetilde{\mathcal{P}}_{\kappa}$
solves equation \eqref{eq:moderate-posterior}. By letting  
\begin{equation}
J\left(\dot{P}||P^{\prime\prime}\rightsquigarrow Q;\kappa\right)=\inf_{P^{\prime}\in\dot{\mathcal{P}}}I\left(P^{\prime},\dot{P}||P^{\prime\prime}\rightsquigarrow Q;\kappa\right)\label{eq:worst-case-gain}
\end{equation}
for any $P^{\prime\prime}\in\ddot{\mathcal{P}}$, that set may be
written as
\begin{equation}
\widetilde{\mathcal{P}}_{\kappa}=\left\{ P\in\mathcal{P}:\inf_{P^{\prime\prime}\in\ddot{\mathcal{P}}}J\left(\dot{P}||P^{\prime\prime}\rightsquigarrow P;\kappa\right)=\inf_{P^{\prime\prime}\in\ddot{\mathcal{P}}}\sup_{Q\in\mathcal{P}}J\left(\dot{P}||P^{\prime\prime}\rightsquigarrow Q;\kappa\right)\right\} .\label{eq:candidate-MP-set}
\end{equation}
The \emph{moderate posterior action with caution $\kappa$} is
\begin{equation}
\widetilde{a}_{\kappa}=\arg\inf_{a\in\mathcal{A}}\int L\left(\theta,a\right)d\widetilde{P}_{\kappa}\left(\theta\right),\label{eq:IP}
\end{equation}
which defines making decisions on the basis of the moderate posterior
as taking actions that minimize its expected loss. For example, if
$\ddot{P}$ is the only confidence posterior under consideration,
then $\ddot{\mathcal{P}}=\left\{ \ddot{P}\right\} $ and 

\begin{eqnarray}
 & \widetilde{P}_{\kappa}=\arg\sup_{Q\in\mathcal{P}}\left(\inf_{P\in\dot{\mathcal{P}}_{\kappa}}I\left(P||\ddot{P}\rightsquigarrow Q\right)\right);\label{eq:MP-single-CP}
\end{eqnarray}
\begin{equation}
\dot{\mathcal{P}}_{\kappa}=\left\{ \kappa P^{\prime}+\left(1-\kappa\right)\dot{P}:P^{\prime}\in\dot{\mathcal{P}}\right\} ,\label{eq:revealed-posteriors}
\end{equation}
which recalls equation \eqref{eq:CG-mixture}. Since $\dot{P}\in\dot{\mathcal{P}}_{\kappa}\subseteq\dot{\mathcal{P}}$,
$\dot{\mathcal{P}}_{0}=\left\{ \dot{P}\right\} $, and $\dot{\mathcal{P}}_{1}=\dot{\mathcal{P}}$,
the effect of $\kappa<1$ as opposed to $\kappa=1$ is to replace
the knowledge base $\dot{\mathcal{P}}$ with a subset $\dot{\mathcal{P}}_{\kappa}$
containing the working Bayesian posterior $\dot{P}$ \citep[cf.][]{RefWorks:1618}.

The two extreme cases of caution reduce decision making to previous
frameworks. A complete lack of caution $\left(\kappa=0\right)$ leads
to the sole use of the working Bayesian posterior for the minimization
of posterior expected loss: $\widetilde{P}_{0}=\dot{P}$. On the other
hand, complete caution $\left(\kappa=1\right)$ leads to ignoring
the working Bayesian posterior and, in the case of a single confidence
posterior, to the framework of \citet{continuum}, in which $\widetilde{P}_{1}$
is called the \emph{blended posterior}.

\subsection{Minimum information divergence}

The following fact is from \citet{ISI:A1979GQ08900002}; see also
\citet{Phuber77}, \citet{Harremoes2007}, \citet{continuum}, and
especially \citet{Topsoe2007b}.
\begin{lem}
\label{lem:equilibrium}Given a distribution $\ddot{P}$ on $\left(\Theta,\mathcal{H}\right)$,
if $\dot{\mathcal{P}}$ is convex and $I\left(P||\ddot{P}\right)<\infty$
for all $P\in\dot{\mathcal{P}}$, then 
\[
\sup_{Q\in\mathcal{P}}\inf_{P^{\prime}\in\dot{\mathcal{P}}}I\left(P^{\prime}||\ddot{P}\rightsquigarrow Q\right)=\inf_{Q\in\dot{\mathcal{P}}}I\left(Q||\ddot{P}\right).
\]

\end{lem}
The resulting theorem is useful for finding $\widetilde{P}_{\kappa}$: 
\begin{thm}
\label{thm:equilibrium}For any $\kappa\in\left[0,1\right]$, if $\dot{\mathcal{P}}$
is convex and $I\left(P^{\prime}||P^{\prime\prime}\right)<\infty$
for all $P^{\prime}\in\dot{\mathcal{P}}$ and $P^{\prime\prime}\in\ddot{\mathcal{P}}$,
then the moderate posterior $\widetilde{P}_{\kappa}$ is given by
equations \eqref{eq:moderate-posterior-unique} and \eqref{eq:revealed-posteriors}
with
\begin{equation}
\widetilde{\mathcal{P}}_{\kappa}=\left\{ P\in\mathcal{P}:\inf_{P^{\prime\prime}\in\ddot{\mathcal{P}}}I\left(P||P^{\prime\prime}\right)=\inf_{P^{\prime\prime}\in\ddot{\mathcal{P}}}\inf_{Q\in\dot{\mathcal{P}}_{\kappa}}I\left(Q||P^{\prime\prime}\right)\right\} .\label{eq:candidate-MPs}
\end{equation}
\end{thm}
\begin{proof}
For any $\kappa\in\left[0,1\right]$ and $P^{\prime\prime}\in\ddot{\mathcal{P}}$,
equations \eqref{eq:worst-case-gain} and \eqref{eq:MP-single-CP},
with Lemma \ref{lem:equilibrium}, imply
\begin{eqnarray*}
\sup_{Q\in\mathcal{P}}J\left(\dot{P}||P^{\prime\prime}\rightsquigarrow Q;\kappa\right) & = & \sup_{Q\in\mathcal{P}}\inf_{P\in\dot{\mathcal{P}}_{\kappa}}I\left(P||P^{\prime\prime}\rightsquigarrow Q\right),\\
 & = & \inf_{Q\in\dot{\mathcal{P}}_{\kappa}}I\left(Q||P^{\prime\prime}\right)
\end{eqnarray*}
and thus
\[
\inf_{P^{\prime\prime}\in\ddot{\mathcal{P}}}\sup_{Q\in\mathcal{P}}J\left(\dot{P}||P^{\prime\prime}\rightsquigarrow Q;\kappa\right)=\inf_{P^{\prime\prime}\in\ddot{\mathcal{P}}}\inf_{Q\in\dot{\mathcal{P}}_{\kappa}}I\left(Q||P^{\prime\prime}\right).
\]
Equation \eqref{eq:candidate-MP-set} thereby reduces to equation
\eqref{eq:candidate-MPs}.
\end{proof}
An immediate consequence is 
\begin{cor}
\label{cor:single-MP}Given $\ddot{\mathcal{P}}=\left\{ \ddot{P}\right\} $,
if $\dot{\mathcal{P}}$ is convex and $I\left(P^{\prime}||\ddot{P}\right)<\infty$
for all $P^{\prime}\in\dot{\mathcal{P}}$, then the moderate posterior
is
\begin{equation}
\widetilde{P}_{\kappa}=\arg\inf_{Q\in\dot{\mathcal{P}}_{\kappa}}I\left(Q||\ddot{P}\right).\label{eq:maxent-sho}
\end{equation}

\end{cor}
Inference is also simplified when at least one of the confidence posteriors
is sufficiently close to the working Bayesian posterior:
\begin{cor}
\label{cor:MP-is-benchmark}If $\ddot{\mathcal{P}}\cap\dot{\mathcal{P}}_{\kappa}$
is nonempty, $\dot{\mathcal{P}}$ is convex, and $I\left(P^{\prime}||P^{\prime\prime}\right)<\infty$
for all $P^{\prime}\in\dot{\mathcal{P}}$ and $P^{\prime\prime}\in\ddot{\mathcal{P}}$,
then the moderate posterior is the confidence posterior that is closest
to the working Bayesian posterior in the sense that 
\begin{equation}
\widetilde{P}_{\kappa}=\arg\inf_{P^{\prime\prime}\in\ddot{\mathcal{P}}\cap\dot{\mathcal{P}}_{\kappa}}I\left(\dot{P}||P^{\prime\prime}\right).\label{eq:MP-is-benchmark}
\end{equation}
\end{cor}
\begin{proof}
For any $P^{\prime\prime}\in\ddot{\mathcal{P}}\cap\dot{\mathcal{P}}_{\kappa}$,
equation \eqref{eq:cross-entropy} implies both $\inf_{Q\in\dot{\mathcal{P}}_{\kappa}}I\left(Q||P^{\prime\prime}\right)=I\left(P^{\prime\prime}||P^{\prime\prime}\right)=0$
and $I\left(P||P^{\prime\prime}\right)>0$ for all $P\in\mathcal{P}\backslash\left(\ddot{\mathcal{P}}\cap\dot{\mathcal{P}}_{\kappa}\right)$
\citep[e.g.,][]{CoverThomas1991}. Thus, by Theorem \ref{thm:equilibrium},
\[
\widetilde{\mathcal{P}}_{\kappa}=\left\{ P\in\mathcal{P}:I\left(P||P^{\prime\prime}\right)=0,P^{\prime\prime}\in\ddot{\mathcal{P}}\cap\dot{\mathcal{P}}_{\kappa}\right\} =\ddot{\mathcal{P}}\cap\dot{\mathcal{P}}_{\kappa},
\]
which was assumed to be nonempty. Equation \eqref{eq:MP-is-benchmark}
then follows from equation \eqref{eq:moderate-posterior-unique}.\end{proof}
\begin{rem}
\label{rem:no-caution-dependence}Unless $\kappa=0$, the condition
that $\ddot{\mathcal{P}}\cap\dot{\mathcal{P}}_{\kappa}$ be nonempty
holds whenever the plausible posteriors are sufficiently unrestricted.
The most important such setting for applications is a complete lack
of constraints $\left(\dot{\mathcal{P}}=\mathcal{P}\right)$, in which
case 
\[
\ddot{\mathcal{P}}\cap\dot{\mathcal{P}}_{\kappa}=\ddot{\mathcal{P}}\cap\left\{ \kappa P+\left(1-\kappa\right)\dot{P}:P\in\mathcal{P}\right\} 
\]
and, if $\mathcal{P}$ is convex and unbounded, then $\ddot{\mathcal{P}}\cap\dot{\mathcal{P}}_{\kappa}=\ddot{\mathcal{P}}\cap\mathcal{P}=\ddot{\mathcal{P}}$
for any $\kappa\in\left(0,1\right]$.
\end{rem}

\section{\label{sec:Examples}Examples}

The first two examples involve the continuous, scalar parameters typical
of point and interval estimation $\left(\Theta=\mathbb{R}\right)$.
For simplicity, each uses only a single confidence posterior $\left(\ddot{\mathcal{P}}=\left\{ \ddot{P}\right\} \right)$.
\begin{example}
\label{exa:normal-normal}$P_{\theta,1}$ is the normal distribution
of mean $\theta$ and variance 1, i.e., $X\sim\N\left(\theta,1\right)$,
and $X=x$ is observed. Further, $\theta\sim\N\left(\mu,\sigma^{2}\right)$
with unknown $\mu$ and $\sigma$ of known lower and upper bounds:
$\mu\in\left[\underline{\mu},\overline{\mu}\right]$; $\sigma\in\left[\underline{\sigma},\overline{\sigma}\right]$.
For generality, the bounds are extended real numbers: $\underline{\mu}\in\left\{ -\infty\right\} \cup\mathbb{R}$,
$\underline{\sigma}\in\left[0,\infty\right)$, $\overline{\mu}=\mathbb{R}\cup\left\{ \infty\right\} $,
$\overline{\sigma}=\left[0,\infty\right)\cup\left\{ \infty\right\} $.
The intervals $\left[\underline{\mu},\overline{\mu}\right]$ and $\left[\underline{\sigma},\overline{\sigma}\right]$
are open only as required to ensure that $\mu,\log\sigma\in\mathbb{R}$
in the presence of infinite bounds or $\underline{\sigma}=0$, e.g.,
$\underline{\mu}=0,\overline{\mu}=\infty\implies\left[\underline{\mu},\overline{\mu}\right]=\left[0,\infty\right)$.
The working prior is $\N\left(\dot{\mu},\dot{\sigma}^{2}\right)$
for some given $\dot{\mu}\in\left[\underline{\mu},\overline{\mu}\right]$
and $\dot{\sigma}\in\left[\underline{\sigma},\overline{\sigma}\right]$.
By Bayes's theorem \citep[e.g.,][]{CarlinLouis3}, 
\[
\dot{P}=\N\left(\frac{\dot{\mu}+\dot{\sigma}^{2}x}{1+\dot{\sigma}^{2}},\frac{\dot{\sigma}^{2}}{1+\dot{\sigma}^{2}}\right);
\]
\[
\dot{\mathcal{P}}=\left\{ \N\left(\frac{\mu+\sigma^{2}x}{1+\sigma^{2}},\frac{\sigma^{2}}{1+\sigma^{2}}\right):\mu\in\left[\underline{\mu},\overline{\mu}\right],\sigma\in\left[\underline{\sigma},\overline{\sigma}\right]\right\} .
\]
By contrast, $\ddot{P}$ is $\N\left(x,1\right)$, not depending on
any prior. This $\ddot{P}$ is a genuine confidence posterior (§\ref{sub:Confidence-posteriors}),
as can be verified from the fact that $\ddot{P}\left(\ddot{\theta}\le\theta\right)=P_{\theta,1}\left(X\ge x\right)$
for all $\theta\in\mathbb{R}$ and $x\in\mathbb{R}$. The performance
of any estimator $\widehat{\theta}$ of $\theta$ may be quantified
by its squared-error prediction loss: $L\left(\theta,\widehat{\theta}\right)=\left(\widehat{\theta}-\theta\right)^{2}$.
By equation \eqref{eq:CG}, the $\kappa$CG estimate is
\begin{eqnarray*}
\dot{a}_{\kappa} & = & \arg\inf_{\widehat{\theta}\in\mathbb{R}}\left(\kappa\sup_{P^{\prime}\in\dot{\mathcal{P}}}\int\left(\widehat{\theta}-\theta\right)^{2}dP^{\prime}\left(\theta\right)+\left(1-\kappa\right)\int\left(\widehat{\theta}-\theta\right)^{2}d\dot{P}\left(\theta\right)\right)\\
 & = & \arg\inf_{\widehat{\theta}\in\mathbb{R}}\left(\kappa\left[\int\left(\widehat{\theta}-\theta\right)^{2}dP_{\mu\left(\widehat{\theta}\right),\underline{\sigma}}\left(\theta\right)\right]+\left(1-\kappa\right)\int\left(\widehat{\theta}-\theta\right)^{2}d\dot{P}\left(\theta\right)\right),
\end{eqnarray*}
where $\mu\left(\widehat{\theta}\right)=\underline{\mu}$ if $\left|\widehat{\theta}-\underline{\mu}\right|>\left|\widehat{\theta}-\overline{\mu}\right|$
and $\mu\left(\widehat{\theta}\right)=\overline{\mu}$ otherwise;
$P_{\mu\left(\widehat{\theta}\right),\underline{\sigma}}=\N\left(\mu\left(\widehat{\theta}\right),\underline{\sigma}^{2}\right)$.
If $\underline{\sigma}=0$, $P_{\mu\left(\widehat{\theta}\right),0}$
is the Dirac measure at $\mu\left(\widehat{\theta}\right)$, implying
that
\[
\dot{a}_{\kappa}=\arg\inf_{\widehat{\theta}\in\mathbb{R}}\left(\kappa\left(\widehat{\theta}-\mu\left(\widehat{\theta}\right)\right)^{2}+\left(1-\kappa\right)\int\left(\widehat{\theta}-\theta\right)^{2}d\dot{P}\left(\theta\right)\right),
\]
which only has a solution if $\underline{\mu}>-\infty$ and $\overline{\mu}<\infty$.
Those restrictions are not needed for the estimate based on the moderate
posterior. Since 
\begin{equation}
\dot{\mathcal{P}}_{\kappa}=\left\{ \kappa\N\left(\frac{\mu+\sigma^{2}x}{1+\sigma^{2}},\frac{\sigma^{2}}{1+\sigma^{2}}\right)+\left(1-\kappa\right)\dot{P}:\mu\in\left[\underline{\mu},\overline{\mu}\right],\sigma\in\left[\underline{\sigma},\overline{\sigma}\right]\right\} ,\label{eq:normal-normal-revealed}
\end{equation}
Corollary \ref{cor:single-MP} entails
\begin{eqnarray*}
\widetilde{P}_{\kappa} & = & \arg\inf_{\mu\in\left[\underline{\mu},\overline{\mu}\right],\sigma\in\left[\underline{\sigma},\overline{\sigma}\right]}I\left(\N\left(\frac{\mu+\sigma^{2}x}{1+\sigma^{2}},\frac{\sigma^{2}}{1+\sigma^{2}}\right)||\N\left(x,1\right)\right)\\
 & = & \arg\inf_{\mu\in\left[\underline{\mu},\overline{\mu}\right],\sigma\in\left[\underline{\sigma},\overline{\sigma}\right]}\left(\log\frac{1+\sigma^{2}}{\sigma^{2}}+\frac{\sigma^{2}}{1+\sigma^{2}}+\left(\frac{\mu-x}{1+\sigma^{2}}\right)^{2}\right),
\end{eqnarray*}
with the second equality from, e.g., \citet[p. 189]{RefWorks:292}.
Substituting $\widetilde{P}_{\kappa}$ into equation \eqref{eq:IP}
gives the moderate-posterior-mean as the estimate of $\theta$:
\[
\widetilde{a}_{\kappa}=\arg\inf_{\widehat{\theta}\in\mathbb{R}}\int\left(\widehat{\theta}-\theta\right)^{2}d\widetilde{P}_{\kappa}\left(\theta\right)=\int\theta d\widetilde{P}_{\kappa}\left(\theta\right),
\]
which is unique even if $\underline{\mu}=-\infty$, $\overline{\mu}=\infty$,
$\underline{\sigma}=0$, and $\overline{\sigma}=\infty$.
\end{example}
The next example drops the parametric assumptions about the plausible
prior distributions.
\begin{example}
\label{exa:ambiguous-normal}$X\sim\N\left(\theta,1\right)$ with
no information about $\theta$ except that $\theta\in\mathbb{R}=\Theta$,
that $X=x$ is observed, and that $\dot{P}$ is the working Bayesian
posterior distribution of $\theta$. It follows that $\dot{\mathcal{P}}$
is the set of all distributions on the Borel space $\left(\mathbb{R},\mathcal{B}\left(\mathbb{R}\right)\right)$.
Again under quadratic loss, by equation \eqref{eq:CG}, the $\kappa$CG
estimate is
\[
\dot{a}_{\kappa}=\arg\inf_{\widehat{\theta}\in\mathbb{R}}\left(\kappa\sup_{P^{\prime}\in\dot{\mathcal{P}}}\int\left(\widehat{\theta}-\theta\right)^{2}dP^{\prime}\left(\theta\right)+\left(1-\kappa\right)\int\left(\widehat{\theta}-\theta\right)^{2}d\dot{P}\left(\theta\right)\right),
\]
which is the posterior mean $\int\theta d\dot{P}\left(\theta\right)$
if $\kappa=0$ but which has no unique value for any other value of
$\kappa$ since $\sup_{P^{\prime}\in\dot{\mathcal{P}}}\int\left(\widehat{\theta}-\theta\right)^{2}dP^{\prime}\left(\theta\right)=\infty$
for any $\widehat{\theta}$. By contrast, equation \eqref{eq:IP}
specifies the unique moderate-posterior estimate given $\ddot{P}=\N\left(x,1\right)$:
\[
\widetilde{a}_{\kappa}=\arg\inf_{\widehat{\theta}\in\mathbb{R}}\int\left(\widehat{\theta}-\theta\right)^{2}d\widetilde{P}_{\kappa}\left(\theta\right)=\int\theta d\widetilde{P}_{\kappa}\left(\theta\right),
\]
where, provided that $\kappa>0$, $\widetilde{P}_{\kappa}=\ddot{P}$
according to Corollary \ref{cor:MP-is-benchmark} since $\ddot{P}\in\mathcal{P}=\dot{\mathcal{P}}_{\kappa}$,
leading to
\[
\widetilde{a}_{\kappa}=\int\theta d\ddot{P}\left(\theta\right),
\]
the frequentist posterior mean.
\end{example}
The last example involves a discrete focus parameter, as is typical
of hypothesis testing and model selection applications.
\begin{example}
\label{exa:single-test}Consider the indicator parameter $\theta$
defined such that $\theta=0$ if the null hypothesis about $\theta_{\ast\ast}$
is true $\left(\theta_{\ast\ast}=0\right)$ and $\theta=1$ if the
alternative hypothesis about $\theta_{\ast\ast}$ is true $\left(\theta_{\ast\ast}\ne0\right)$.
Equivalently, in terms of $\theta_{\ast}=\left|\theta_{\ast\ast}\right|$,
$\theta=0$ if $\theta_{\ast}=0$ and $\theta=1$ if $\theta_{\ast}>0$.
If $\dot{P}$ is a working Bayesian posterior for $\theta_{\ast\ast}$,
then $\dot{P}\left(\dot{\theta}=0\right)$ is the corresponding working
Bayesian posterior probability that the null hypothesis is true. Let
$p^{\left(1\right)}$ and $p^{\left(2\right)}$ denote observed p-values
of the one-sided test of $\theta_{\ast}=0$ versus $\theta_{\ast}>0$
and thus of the two-sided test of $\theta_{\ast\ast}=0$ versus $\theta_{\ast\ast}\ne0$.
In this example, $p^{\left(1\right)}\left(x\right)\le p^{\left(2\right)}\left(x\right)$,
perhaps because $p^{\left(2\right)}\left(x\right)$ is based on a
test that makes weaker parametric assumptions than that of $p^{\left(1\right)}\left(x\right)$.
For $i=1,2$, let $\ddot{P}_{\ast}^{\left(i\right)}$ denote the
confidence posterior for $\theta_{\ast}$ defined given some $x\in\mathcal{X}$
such that
\[
\ddot{P}_{\ast}^{\left(i\right)}\left(\ddot{\theta}_{\ast}\le\theta_{\ast}\right)=P_{\theta_{\ast},\lambda_{\ast}}\left(p^{\left(i\right)}\left(X\right)\le p^{\left(i\right)}\left(x\right)\right)
\]
for all $\theta_{\ast}\in\Theta_{\ast}$ and $\lambda_{\ast}\in\Lambda_{\ast}$,
where the dependence of $\ddot{P}_{\ast}^{\left(i\right)}$ on $x$
is suppressed, in Section \ref{sub:Confidence-posteriors}. Since
$p^{\left(i\right)}\left(X\right)\sim\unif$$\left(0,1\right)$ under
the null hypothesis that $\theta_{\ast}=0$, it follows that 
\begin{equation}
\ddot{P}_{\ast}^{\left(i\right)}\left(\ddot{\theta}_{\ast}\le0\right)=\ddot{P}_{\ast}^{\left(i\right)}\left(\ddot{\theta}_{\ast}=0\right)=P_{0,\lambda_{\ast}}\left(p^{\left(i\right)}\left(X\right)\le p^{\left(i\right)}\left(x\right)\right)=p^{\left(i\right)}\left(x\right),\label{eq:confidence-0}
\end{equation}
i.e., the confidence posterior probability of the null hypothesis
is equal to the p-value \citep{smallScale,continuum}; cf. \citet{RefWorks:1369}.
With $\ddot{\theta}=0$ if $\ddot{\theta}_{\ast}=0$ and $\ddot{\theta}=1$
if $\ddot{\theta}_{\ast}\ne0$, equation \eqref{eq:confidence-0}
yields $\ddot{P}^{\left(i\right)}\left(\ddot{\theta}=0\right)=p^{\left(i\right)}\left(x\right)$.
From widely applicable conditions for two-sided hypothesis testing
\citep{RefWorks:1218,continuum} and with some $\underline{\dot{P}}_{\emptyset}^{\pri}\in\left(0,1\right)$
given as the lower bound of the prior probabilities of the null hypothesis
and the restriction that no such probability is 1, the knowledge base
is
\[
\dot{\mathcal{P}}=\left\{ P^{\prime}\in\mathcal{P}:\underline{\dot{P}}\left(\dot{\underline{\theta}}=0\right)=\underline{\dot{P}}\left(\left\{ 0\right\} \right)\le P^{\prime}\left(\left\{ 0\right\} \right)<1\right\} ,
\]
the set of plausible posteriors, the distributions on $\left(\left\{ 0,1\right\} ,2^{\left\{ 0,1\right\} }\right)$
with 
\[
\underline{\dot{P}}_{\emptyset}=\underline{\dot{P}}\left(\dot{\underline{\theta}}=0\right)=\left(1+\left(\frac{1-\underline{\dot{P}}_{\emptyset}^{\pri}}{\underline{\dot{P}}_{\emptyset}^{\pri}ep^{\left(2\right)}\left(x\right)\log\left[1/p^{\left(2\right)}\left(x\right)\right]}\right)\right)^{-1}\wedge\underline{\dot{P}}_{\emptyset}^{\pri}
\]
as the lower bound of the plausible posterior probability of the null
hypothesis, where $\dot{\underline{\theta}}\sim\underline{\dot{P}}$.
That lower bound is the greater of the two lower bounds found by separately
applying the methodology of \citet{RefWorks:1218} to $p^{\left(1\right)}\left(x\right)$
and $p^{\left(2\right)}\left(x\right)$. (The binary operator $\wedge$
in the above equation means {}``the minimum of,'' and $\vee$ will
similarly stand for {}``the maximum of.'') Since Theorem \ref{thm:equilibrium}
applies, the moderate posterior $\widetilde{P}_{\kappa}$ is given
by equation \eqref{eq:moderate-posterior-unique} with
\begin{eqnarray*}
\widetilde{\mathcal{P}}_{\kappa} & = & \left\{ P\in\mathcal{P}:I\left(P||\ddot{P}^{\left(1\right)}\right)\wedge I\left(P||\ddot{P}^{\left(2\right)}\right)=I\left(\widetilde{P}_{\kappa}^{\left(1\right)}||\ddot{P}^{\left(1\right)}\right)\wedge I\left(\widetilde{P}_{\kappa}^{\left(2\right)}||\ddot{P}^{\left(2\right)}\right)\right\} \\
 & = & \left\{ \widetilde{P}_{\kappa}^{\left(i\right)}:I\left(\widetilde{P}_{\kappa}^{\left(i\right)}||\ddot{P}^{\left(i\right)}\right)=I\left(\widetilde{P}_{\kappa}^{\left(1\right)}||\ddot{P}^{\left(1\right)}\right)\wedge I\left(\widetilde{P}_{\kappa}^{\left(2\right)}||\ddot{P}^{\left(2\right)}\right),i\in\left\{ 1,2\right\} \right\} ,
\end{eqnarray*}
where $\widetilde{P}_{\kappa}^{\left(i\right)}=\arg\inf_{Q\in\dot{\mathcal{P}}_{\kappa}}I\left(Q||\ddot{P}^{\left(i\right)}\right);$
$\dot{\mathcal{P}}_{\kappa}=\left\{ \kappa P^{\prime}+\left(1-\kappa\right)\dot{P}:P^{\prime}\in\mathcal{P},\underline{\dot{P}}_{\emptyset}\le P^{\prime}\left(\theta=0\right)<1\right\} $.
More simply, 
\[
\widetilde{P}_{\kappa}=\begin{cases}
\widetilde{P}_{\kappa}^{\left(1\right)} & \text{if }I\left(\widetilde{P}_{\kappa}^{\left(1\right)}||\ddot{P}^{\left(1\right)}\right)<I\left(\widetilde{P}_{\kappa}^{\left(2\right)}||\ddot{P}^{\left(2\right)}\right)\\
\widetilde{P}_{\kappa}^{\left(2\right)} & \text{if }I\left(\widetilde{P}_{\kappa}^{\left(1\right)}||\ddot{P}^{\left(1\right)}\right)>I\left(\widetilde{P}_{\kappa}^{\left(2\right)}||\ddot{P}^{\left(2\right)}\right)\\
\widetilde{P}_{\kappa}^{\left(1\right)} & \text{if }I\left(\widetilde{P}_{\kappa}^{\left(1\right)}||\ddot{P}^{\left(1\right)}\right)=I\left(\widetilde{P}_{\kappa}^{\left(2\right)}||\ddot{P}^{\left(2\right)}\right)\text{ and }I\left(\dot{P}||\widetilde{P}_{\kappa}^{\left(1\right)}\right)\le I\left(\dot{P}||\widetilde{P}_{\kappa}^{\left(2\right)}\right)\\
\widetilde{P}_{\kappa}^{\left(2\right)} & \text{if }I\left(\widetilde{P}_{\kappa}^{\left(1\right)}||\ddot{P}^{\left(1\right)}\right)=I\left(\widetilde{P}_{\kappa}^{\left(2\right)}||\ddot{P}^{\left(2\right)}\right)\text{ and }I\left(\dot{P}||\widetilde{P}_{\kappa}^{\left(1\right)}\right)\ge I\left(\dot{P}||\widetilde{P}_{\kappa}^{\left(2\right)}\right).
\end{cases}
\]
Letting $\dot{P}_{\emptyset}=\dot{P}\left(\dot{\theta}=0\right)$
and letting $\widetilde{\theta}$ denote the focus parameter according
to the moderate posterior $\left(\widetilde{\theta}\sim\widetilde{P}_{\kappa}\right)$,
\begin{eqnarray}
\widetilde{P}_{\kappa}^{\left(i\right)} & = & \arg\inf_{Q\in\dot{\mathcal{P}}_{\kappa}}\sum_{j=0,1}Q\left(\theta=j\right)\log\frac{Q\left(\theta=j\right)}{\ddot{P}^{\left(i\right)}\left(\theta=j\right)}\nonumber \\
\widetilde{P}_{\kappa}^{\left(i\right)}\left(\widetilde{\theta}=0\right) & = & \arg\inf_{Q_{\emptyset}\in\left\{ \kappa P_{\emptyset}+\left(1-\kappa\right)\dot{P}_{\emptyset}:P_{\emptyset}\in\left[\underline{\dot{P}}_{\emptyset},1\right)\right\} }Q_{\emptyset}\log\frac{Q_{\emptyset}}{p^{\left(i\right)}\left(x\right)}+\left(1-Q_{\emptyset}\right)\log\frac{1-Q_{\emptyset}}{1-p^{\left(i\right)}\left(x\right)}\nonumber \\
 & = & \arg\inf_{Q_{\emptyset}\in\left[\kappa\underline{\dot{P}}_{\emptyset}+\left(1-\kappa\right)\dot{P}_{\emptyset},\kappa+\left(1-\kappa\right)\dot{P}_{\emptyset}\right)}Q_{\emptyset}\log\frac{Q_{\emptyset}}{p^{\left(i\right)}\left(x\right)}+\left(1-Q_{\emptyset}\right)\log\frac{1-Q_{\emptyset}}{1-p^{\left(i\right)}\left(x\right)}\nonumber \\
 & = & \begin{cases}
\kappa\underline{\dot{P}}_{\emptyset}+\left(1-\kappa\right)\dot{P}_{\emptyset} & \text{if }p^{\left(i\right)}\left(x\right)<\kappa\underline{\dot{P}}_{\emptyset}+\left(1-\kappa\right)\dot{P}_{\emptyset}\\
p^{\left(i\right)}\left(x\right) & \text{if }\kappa\underline{\dot{P}}_{\emptyset}+\left(1-\kappa\right)\dot{P}_{\emptyset}\le p^{\left(i\right)}\left(x\right)\le\kappa+\left(1-\kappa\right)\dot{P}_{\emptyset}\\
\kappa+\left(1-\kappa\right)\dot{P}_{\emptyset} & \text{if }p^{\left(i\right)}\left(x\right)>\kappa+\left(1-\kappa\right)\dot{P}_{\emptyset}.
\end{cases}\label{eq:simple-hypothesis}
\end{eqnarray}
Since $\widetilde{P}_{\kappa}\in\widetilde{\mathcal{P}}_{\kappa}$,
\begin{equation}
\widetilde{P}_{\kappa}\left(\widetilde{\theta}=0\right)\in\begin{cases}
\left\{ \kappa\underline{\dot{P}}_{\emptyset}+\left(1-\kappa\right)\dot{P}_{\emptyset}\right\}  & \text{if }p^{\left(1\right)}\left(x\right)\le p^{\left(2\right)}\left(x\right)<\kappa\underline{\dot{P}}_{\emptyset}+\left(1-\kappa\right)\dot{P}_{\emptyset}\\
\left\{ p^{\left(2\right)}\left(x\right)\right\}  & \text{if }p^{\left(1\right)}\left(x\right)<\kappa\underline{\dot{P}}_{\emptyset}+\left(1-\kappa\right)\dot{P}_{\emptyset}\le p^{\left(2\right)}\left(x\right)\le\kappa+\left(1-\kappa\right)\dot{P}_{\emptyset}\\
\left\{ p^{\left(1\right)}\left(x\right),p^{\left(2\right)}\left(x\right)\right\}  & \text{if }\kappa\underline{\dot{P}}_{\emptyset}+\left(1-\kappa\right)\dot{P}_{\emptyset}\le p^{\left(1\right)}\left(x\right)\le p^{\left(2\right)}\left(x\right)\le\kappa+\left(1-\kappa\right)\dot{P}_{\emptyset}\\
\left\{ p^{\left(1\right)}\left(x\right)\right\}  & \text{if }\kappa\underline{\dot{P}}_{\emptyset}+\left(1-\kappa\right)\dot{P}_{\emptyset}\le p^{\left(1\right)}\left(x\right)\le\kappa+\left(1-\kappa\right)\dot{P}_{\emptyset}<p^{\left(2\right)}\left(x\right)\\
\left\{ \kappa+\left(1-\kappa\right)\dot{P}_{\emptyset}\right\}  & \text{if }p^{\left(2\right)}\left(x\right)\ge p^{\left(1\right)}\left(x\right)>\kappa+\left(1-\kappa\right)\dot{P}_{\emptyset},
\end{cases}\label{eq:simple-hypothesis-p-set}
\end{equation}
from which the extreme condition $p^{\left(1\right)}\left(x\right)<\kappa\underline{\dot{P}}_{\emptyset}+\left(1-\kappa\right)\dot{P}_{\emptyset}<\kappa+\left(1-\kappa\right)\dot{P}_{\emptyset}<p^{\left(2\right)}\left(x\right)$
is omitted for brevity. In the case of no caution, the working Bayesian
posterior probability is recovered: $\widetilde{P}_{0}\left(\widetilde{\theta}=0\right)=\dot{P}\left(\dot{\theta}=0\right)$,
which does not depend on $p\left(x\right)$. More interestingly, the
case of complete caution leads to 
\begin{equation}
\widetilde{P}_{1}\left(\widetilde{\theta}=0\right)\in\begin{cases}
\left\{ \underline{\dot{P}}\left(\dot{\underline{\theta}}=0\right)\right\}  & \text{if }p^{\left(1\right)}\left(x\right)\le p^{\left(2\right)}\left(x\right)<\underline{\dot{P}}_{\emptyset}\\
\left\{ p^{\left(2\right)}\left(x\right)\right\}  & \text{if }p^{\left(1\right)}\left(x\right)<\underline{\dot{P}}_{\emptyset}\le p^{\left(2\right)}\left(x\right)\\
\left\{ p^{\left(1\right)}\left(x\right),p^{\left(2\right)}\left(x\right)\right\}  & \text{if }\underline{\dot{P}}_{\emptyset}\le p^{\left(1\right)}\left(x\right)\le p^{\left(2\right)}\left(x\right),
\end{cases}\label{eq:extended-blend}
\end{equation}
which has no dependence on $\dot{P}\left(\dot{\theta}=0\right)$.
The simplifying effect of considering only a single p-value is evident
from using $p^{\left(1\right)}\left(x\right)=p^{\left(2\right)}\left(x\right)$
in the formulas \eqref{eq:simple-hypothesis-p-set} and \eqref{eq:extended-blend}.
For example, expression \eqref{eq:extended-blend} results in a unique
$\widetilde{P}_{1}\left(\widetilde{\theta}=0\right)$ equal to the
blended posterior probability of \citet{continuum}. When formulas
\eqref{eq:simple-hypothesis-p-set} and \eqref{eq:extended-blend}
say no more than $\widetilde{P}_{\kappa}\left(\widetilde{\theta}=0\right)\in\left\{ p^{\left(1\right)}\left(x\right),p^{\left(2\right)}\left(x\right)\right\} $,
equation \eqref{eq:moderate-posterior-unique} ensures the uniqueness
of the moderate posterior probability by equating it with the p-value
closest to the working Bayesian posterior probability:
\begin{eqnarray*}
\widetilde{P}_{\kappa} & = & \arg\inf_{P^{\prime\prime\prime}\in\left\{ p^{\left(1\right)}\left(x\right),p^{\left(2\right)}\left(x\right)\right\} }I\left(\dot{P}||P^{\prime\prime\prime}\right)=\ddot{P}^{\left(\widetilde{\imath}\right)};\\
\widetilde{\imath} & = & \arg\inf_{i\in\left\{ 1,2\right\} }\left(\dot{P}\left(\dot{\theta}=0\right)\log\frac{\dot{P}\left(\theta=0\right)}{p^{\left(i\right)}\left(x\right)}+\dot{P}\left(\dot{\theta}=1\right)\log\frac{\dot{P}\left(\dot{\theta}=1\right)}{1-p^{\left(i\right)}\left(x\right)}\right),
\end{eqnarray*}
which is a special case of Corollary \ref{cor:MP-is-benchmark}. In
this way, the caution parameter, the working Bayesian posterior, and
the constraints on the plausible posteriors together overcome the
dilemma of whether to use the more conservative p-value or the less
conservative p-value.
\end{example}
~

\section{\label{sec:Variations}Extending the caution framework}

\subsection{Variations of the framework}

The above framework for balancing Bayesian and frequentist approaches
to inference does not apply to all situations encountered in applications.
The various permutations of the Bayesian and confidence posteriors
as the \emph{working posterior} $\dot{P}$, used exclusively in the
absence of caution, and a \emph{benchmark posterior} $\ddot{P}$,
over which inference will be improved as much as possible, in equations
\eqref{eq:MP-single-CP} and \eqref{eq:maxent-sho} lead to four versions
of the proposed approach:
\begin{enumerate}
\item $\dot{P}$ is a Bayesian posterior in $\dot{\mathcal{P}}$, and $\ddot{P}$
is a confidence posterior. This version yields the balance between
Bayesian and frequentist inference defined in Section \ref{sec:General-theory}
and illustrated in Section \ref{sec:Examples}.
\item $\dot{P}$ is a confidence posterior, and $\ddot{P}$ is a Bayesian
posterior in $\dot{\mathcal{P}}$. The potential uses of this reversal
are unclear since it would paradoxically lead to dependence on a single
Bayesian posterior to the extent of the caution.
\item $\dot{P}=\ddot{P},$ where $\dot{P}$ is a Bayesian posterior in $\dot{\mathcal{P}}$.
Using the same Bayesian posterior as both the working posterior and
the benchmark posterior is attractive in the absence of reliable confidence
intervals or p-values from which a confidence posterior could be constructed.
Thus, this version extends the scope of the framework across the domains
to which Bayesian methods apply. However, this version becomes trivial
whenever equation \eqref{eq:maxent-sho} holds according to Corollary
\ref{cor:single-MP}, for in that case, $\widetilde{P}_{\kappa}=\arg\inf_{Q\in\dot{\mathcal{P}}_{\kappa}}I\left(Q||\dot{P}\right)=\dot{P}$
for all $\kappa\in\left[0,1\right]$ since $\dot{P}\in\dot{\mathcal{P}}_{\kappa}$
necessarily. In other words, the Bayesian posterior would be used
for inference irrespective of the degree of caution and the knowledge
base.
\item $\dot{P}=\ddot{P},$ where $\ddot{P}$ is a confidence posterior.
Using the same Bayesian posterior as both the working posterior and
the benchmark posterior is useful when a set $\dot{\mathcal{P}}$
of plausible posteriors can be specified but when no member of that
set can be singled out as special. In many cases involving a continuous
parameter $\theta$, no such member can be derived from the knowledge
base $\dot{\mathcal{P}}$ without imposing arbitrary procedures such
as averaging over the members with respect to some measure chosen
for convenience. That will be explained in Section \ref{sub:Inference-without-OP},
where the case of two unequal confidence posteriors will also be considered.
\end{enumerate}
For simplicity, the versions are described as if $\ddot{\mathcal{P}}=\left\{ \ddot{P}\right\} $,
but they also pertain to a set $\ddot{\mathcal{P}}$ of multiple benchmark
posteriors that define the moderate posterior $\widetilde{P}_{\kappa}$
according to equation \eqref{eq:moderate-posterior-unique}. The three
most applicable of those versions are summarized in Table \ref{tab:versions}. 

Generalizing beyond those versions, the working posterior $\dot{P}$
can be any posterior distribution that would be used exclusively in
the absence of caution, whereas when there is caution, inferences
are made with the goal that they improve upon those that would be
made using any other posterior distribution in $\ddot{\mathcal{P}}$.
The application at hand can help determine which of those distributions
is a Bayesian posterior and which is some other type of distribution
such as a confidence distribution. For example, given a working posterior
$\dot{P}$ from a proper prior, a posterior from an improper prior
could be used as the benchmark posterior $\ddot{P}$ in the absence
of a suitable confidence posterior \citep[cf.][]{continuum}.

\begin{table}
\begin{tabular}{|c|c|c|}
\hline 
$\begin{aligned}\,\\
\text{\textbf{Setting}}
\end{aligned}
$ & $\begin{aligned}\,\\
\boldsymbol{\dot{P}}
\end{aligned}
$ & $\begin{aligned}\,\\
\boldsymbol{\ddot{P}}\text{ or }\boldsymbol{\ddot{\mathcal{P}}}
\end{aligned}
$\tabularnewline
\hline 
Bayes and frequentist approaches apply & Bayes posterior & confidence posterior\tabularnewline
\hline 
no confidence intervals or p-values & \multicolumn{2}{c|}{Bayes posterior}\tabularnewline
\hline 
continuous $\theta$ \& only a set of Bayes posteriors & \multicolumn{2}{c|}{confidence posterior}\tabularnewline
\hline 
\end{tabular}

\caption{Settings for three versions of the proposed framework.\label{tab:versions}}
\end{table}

\subsection{\label{sub:Inference-without-OP}Inference without a working Bayesian
posterior}

The more interesting of the two widely applicable variations of the
framework is that in which both $\dot{P}$ and $\ddot{P}$ are the
same confidence posterior. Thus, some implications of using a single
confidence posterior simultaneously as the working posterior and as
the benchmark posterior $\left(\dot{P}=\ddot{P}\right)$ merit noting.
First, complete caution $\left(\kappa=1\right)$ leads to ignoring
the role of the confidence posterior as a working posterior and thereby
collapses to the blended inferences of \citet{continuum}. Second,
when $\kappa<1$ and the confidence posterior is not a plausible posterior
$\left(\ddot{P}\notin\dot{\mathcal{P}}\right)$, the moderate posterior
may not be a plausible posterior $\left(\widetilde{P}_{\kappa}\notin\dot{\mathcal{P}}\right)$.
In fact, whenever sufficient conditions for Corollary \ref{cor:single-MP}
are met, $\widetilde{P}_{\kappa}$ will be plausible only if $\ddot{P}$
is plausible. That suggests the use of $\kappa=1$ in the absence
of a working Bayesian posterior in order to avoid excessive dependence
on $\ddot{P}$ at the expense of $\dot{\mathcal{P}}$, the knowledge
base. On the other hand, allowing $\widetilde{P}_{\kappa}\notin\dot{\mathcal{P}}$
makes $\widetilde{P}_{\kappa}$ less dependent on the precise borders
of $\dot{\mathcal{P}}$, and this may be desirable to the extent that
such borders are uncertain or subjectively specified. Third, when
$\ddot{P}\in\dot{\mathcal{P}}$, the moderate posterior is simply
equal to the confidence posterior $\left(\widetilde{P}_{\kappa}=\ddot{P}\right)$
under the sufficient conditions for equation \eqref{eq:maxent-sho}
given by Corollary \ref{cor:single-MP}. The following examples illustrate
the second and third implications.
\begin{example}
(Variation of Example \ref{exa:ambiguous-normal}.) This example is
trivial since $\dot{P}\in\mathcal{P}$, $\dot{P}=\ddot{P}$, and Corollary
\ref{cor:single-MP} entail $\widetilde{P}_{\kappa}=\ddot{P}$. 
\end{example}
More generally, $\ddot{P}\in\dot{\mathcal{P}}$ and $\dot{P}=\ddot{P}$
imply $\widetilde{P}_{\kappa}=\ddot{P}$ by Corollary \ref{cor:MP-is-benchmark}. 
\begin{example}
(Variation of Example \ref{exa:single-test}.) Because $\ddot{P}\left(\ddot{\theta}=0\right)=p\left(x\right)$,
the identity $\dot{P}=\ddot{P}$ yields $\dot{P}_{\emptyset}=p\left(x\right)$,
reducing equation \eqref{eq:simple-hypothesis} to
\[
\widetilde{P}_{\kappa}\left(\widetilde{\theta}=0\right)=\begin{cases}
\kappa\underline{\dot{P}}_{\emptyset}+\left(1-\kappa\right)p\left(x\right) & \text{if }p\left(x\right)<\underline{\dot{P}}_{\emptyset}\\
p\left(x\right) & \text{if }p\left(x\right)\ge\underline{\dot{P}}_{\emptyset}.
\end{cases}
\]
Re-expressing this as $\widetilde{P}_{\kappa}\left(\widetilde{\theta}=0\right)=\left[\kappa\underline{\dot{P}}_{\emptyset}+\left(1-\kappa\right)p\left(x\right)\right]\vee p\left(x\right)$
facilitates comparison with the equation $\widetilde{P}_{1}\left(\widetilde{\theta}=0\right)=\underline{\dot{P}}_{\emptyset}\vee p\left(x\right)$
used in the blended inference framework \citep{continuum}. Therefore,
$\kappa\in\left[0,1\right)$ entails that $\widetilde{P}_{\kappa}\left(\widetilde{\theta}=0\right)$
is less than the lower bound $\underline{\dot{P}}_{\emptyset}$ whenever
$p\left(x\right)<\underline{\dot{P}}_{\emptyset}$. That would be
clearly unacceptable if $\underline{\dot{P}}_{\emptyset}^{\pri}$
is scientifically established, but if $\underline{\dot{P}}_{\emptyset}^{\pri}$
is instead highly uncertain or subjectively assessed, then $\widetilde{P}_{\kappa}\left(\widetilde{\theta}=0\right)$
can bypass $\underline{\dot{P}}_{\emptyset}$ as warranted. 
\end{example}
~

An alternative to the above approach in the absence of a specified
$\dot{P}$ is to apply the strategy of Section \ref{sec:General-theory}
with $\dot{P}$ as a function of $\dot{\mathcal{P}}$, following \citet{Gajdos200827}.
Examples of functions that transform a set of distributions to a single
distribution include the Steiner point \citep{Gajdos200827}, the
arithmetic mean ({}``center of mass''), and the maximum entropy
distribution \citep{Paris232410}. In the continuous-parameter case,
such functions require a base measure for partitioning.

There is no need to impose an arbitrary base measure if two different
confidence posteriors $\dot{P}$ and $\ddot{P}$ are under consideration
$\left(\dot{P}\ne\ddot{P}\right)$. Using them as the working posterior
and the benchmark posterior in equations \eqref{eq:MP-single-CP}
and \eqref{eq:maxent-sho} would be most appropriate when $\dot{P}$
represents a newer or more risky procedure and when $\ddot{P}$ corresponds
to a better established or more thoroughly tested procedure. More
generally, equation \eqref{eq:moderate-posterior-unique} specifies
how to apply a working confidence posterior $\dot{P}$ with a set
$\ddot{\mathcal{P}}$ of benchmark confidence posteriors.

\section{\label{sec:Discussion}Discussion}

The featured moderate-posterior methodology has been contrasted with
the simpler $\kappa$CG methodology. As Examples \ref{exa:normal-normal}
and \ref{exa:ambiguous-normal} illustrated under quadratic loss,
the former can yield unique actions in a wide variety of settings
in which the latter cannot. Using CG minimaxity $\left(\kappa=1\right)$,
uniqueness has been achieved under quadratic loss by restricting the
action space to finite bounds \citep{Betro1992b} and by similarly
restricting the parameter space $\Theta$ \citep{AbdollahBayati2011}.
The moderate-posterior estimators did not require such restrictions.

The main advantage of the moderate-posterior framework is that it
provides first principles from which a statistician may derive a Bayesian
analysis, a frequentist analysis, or a combination of the two, depending
on the chosen level of caution and on the quality of prior information.
This allows the caution level to be precisely reported with the resulting
statistical inferences. In addition, the caution level may be determined
by the needs of an organization or collaborating scientist rather
than by the personal attitude of the statistician. 

Various factors may be considered in choosing the level of caution.
For example, more caution with Bayesian inference may be warranted
when the confidence posterior represents a frequentist procedure that
has stood the test of time than when it represents a new frequentist
procedure based on questionable assumptions. The caution level could
then be interpreted as the pre-data degree of reluctance an agent
has in modifying the frequentist procedures encoded in the confidence
posterior.

The moderate-posterior framework of Section \ref{sec:General-theory}
is general enough to incorporate conflicting frequentist approaches,
as seen in Example \ref{exa:single-test}. For additional generality,
Section \ref{sub:Inference-without-OP} provides ways to modify the
framework for situations in which any dependence on a subjective or
guessed Bayesian posterior would be undesirable. 

In other situations, any dependence of inference on the level of caution
would be undesirable. Provided that there is at least a little caution,
the use of a sufficiently broad set of plausible posteriors under
the unmodified framework (§\ref{sec:General-theory}) eliminates any
other dependence on the degree of caution (Remark \ref{rem:no-caution-dependence}).

\section*{Acknowledgments}

This research was partially supported  by the Canada Foundation for
Innovation, by the Ministry of Research and Innovation of Ontario,
and by the Faculty of Medicine of the University of Ottawa. 

\begin{flushleft}
\bibliographystyle{elsarticle-harv}
\bibliography{refman}

\par\end{flushleft}

\end{document}